\documentclass[14pt]{article}

\usepackage{mathtext}
\usepackage[T1,T2A]{fontenc}
\usepackage[utf8]{inputenc}
\usepackage[english]{babel}

\usepackage{amsfonts, amsmath, amssymb, amsthm, amscd}
\usepackage{mathtools}
\usepackage{ulem}
\usepackage{array}
\usepackage{geometry}
\usepackage{graphicx}
\usepackage{hyperref}
\usepackage{cancel}
\hypersetup{pdfstartview=FitH,  linkcolor=blue,urlcolor=urlcolor, citecolor=blue, colorlinks=true}

\theoremstyle{theorem}
\newtheorem*{theorem}{Theorem}

\newcommand{\Label}[1]{\label{#1}}

\usepackage{fancyhdr}
\newcommand\RR{\mathbb{R}}

\newcommand{\ZZ}{\mathbb{Z}}
\newcommand{\al}{\alpha}
\newcommand{\Qp}{\mathbb Q_p}

\newcommand{\be}{\beta}

\newcommand{\OO}{\Omega}
\begin{document}

\title{Monotonicity Property of Non-Archimedean Heat Equation}

%%%%% authors:
\author{Alexandra V. Antoniouk\\
Institute of Mathematics, National Academy of Sciences of Ukraine,\\ Tereshchenkivska 3, Kyiv, 01024 Ukraine\\
American University Kyiv, Poshtova Sq 3, 04070, Kyiv, Ukraine\\
E-mail: antoniouk.a@gmail.com
\and
Anatoly N. Kochubei\\
Institute of Mathematics, National Academy of Sciences of Ukraine,\\ Tereshchenkivska 3, Kyiv, 01024 Ukraine \\
E-mail: kochubei@imath.kiev.ua}

\date{}

\maketitle

\bigskip

%%%%%%%%%%%%%%%%%%%%%%%%%%%%%%%%%%%%%%%%%%%%%%%%%%%%%%
\begin{abstract}
We prove the monotone dependence of the heat kernel corresponding to the Vladimirov-Taibleson non-Archimedean fractional differentiation operator $D^\alpha$ on the order $\alpha$.
\end{abstract}

\section{Introduction} \Label{sec1}

Let $K$ be a non-Archimedean local field, $\OO=\left\{ x\in K^n:\ \| x\|_{K^n}\le q^N\right\}$ be an ultrametric ball. Here $N\in \mathbb Z$, $\|(x_1,\ldots x_n)\|_{K^n}=\max\limits_{1\le j\le n}|x_j|_K$, $q$ is the residue cardinality of $K$
(for the basic notions regarding local fields see Section 2 below and \cite{FV, K2001,Weil}).  The best known example of a non-Archimedean local field is the field of $p$-adic numbers; here $p$ is a prime number.

A model example of a pseudo-differential operator acting on real- or complex-valued functions on $\OO$ is the Vladimirov-Taibleson fractional differential operator

\begin{equation}\Label{1-1}\
\big(D^{\al,n} u\big)(x)=c_{n,\al} \int\limits_{K^n} \frac{u(x)-u(y)}{\Vert x-y\Vert_{K^n}^{n+\al}}\, dy, \quad x\in \OO,
\end{equation}
where $dy$ is the Haar measure, $c_{n,\al}=\dfrac{q^{\al}-1}{1-q^{-\al-n}}, \quad \al >0,$ and a function $u$ is extended onto $K^n$ by a nonlocal Dirichlet condition, that is (for the homogeneous case) $u(y)=0$, as $y\in \Omega^c=K^n\setminus \Omega$. The precise understanding of the integral in \eqref{1-1} depends on smoothness assumptions for $u$; the simplest case is that of its boundedness and local constancy when $u(y)=u(x)$, as $y$ belongs to a neighbourhood of a point $x$. Note that $K$ is totally disconnected, and, in contrast to classical analysis, the set of such functions is quite rich.

The choice of the constant $c_{n,\al}$ is motivated by the representation of $D^{\al,n} $ in terms of the Fourier analysis on $K^n$; see \cite{K2001,VVZ}.

The operator \eqref{1-1} and related equations have been studied by many authors, see the monographs \cite{VVZ,K2001,AKS,T,KKZ,Z2016,Z2025} and many recent papers. Its form, especially its nonlocality, resembles that of the fractional Laplacian $(-\Delta)^{\alpha/2}$ of real analysis, though the actual meaning of the objects of a non-Archimedean theory is quite different.

In this note we consider an analog of the first boundary value problem for the heat equation, that is
\begin{equation}\Label{1-2}\
\left\{
\begin{aligned}
\frac{\partial u(t,x)}{\partial t}+\left(D^{\al,n}u\right)(t,x)=0,\quad x\in \Omega,\ t>0;\\
u(t,x)=0,\quad x\in \Omega^c;\\
u(0,x)=\varphi (x),\quad x\in \Omega.
\end{aligned}
\right.
\end{equation}

The non-Archimedean ball $\OO$ is simultaneously open and closed, so that it has no topological boundary. Though one cannot speak, in the classical style, about ``boundary value problems'', there is a well-developed theory of parabolic equations of this kind. See \cite{K2001,AKK2020,Z2016,Z2025,K2018,AKN,BGPW} and references therein. Our aim here is to prove a monotone dependence, in a certain sense, of a solution of (2) on $\al$. A similar problem for the fractional Laplacian was studied in \cite{DGV}.

\section{Local fields \cite{FV,K2001,Weil}}
A (non-Archimedean) local field is a non-discrete totally disconnected locally compact topological field.

An important related notion is that of a field extension. If $k$ is a subfield of a field $K$, then $K$ is called an algebraic extension of $k$. In this situation, the field $K$ is a vector space over $k$, and we call $K$ a finite extension, if this vector space is finite-dimensional. This dimension is called a degree of this extension. A local field $K$ is isomorphic either to a finite algebraic extension of the field $\Qp$ of $p$-adic numbers, if $K$ has characteristic $0$, or to the field of formal Laurent series with coefficients from a finite field, if $K$ has a positive characteristic. By Ostrowski's theorem, $\Qp$ is the only possible alternative to $\RR$ as a completion of the field of rational numbers. This shows the importance of non-Archimedean mathematics as one of the principal branches of mathematical science.

Any local field is endowed with absolute value $\vert \cdot \vert_K$, such that: 1) $\vert x\vert_K=0$, if and only if $x=0$; 2) $\vert xy\vert_K = \vert x\vert_K\cdot\vert y \vert_K$; 3) $\vert x+y\vert_K\leq \max (\vert x\vert_K, \vert y\vert_K)$. The last property is called the ultrametric inequality. It implies the equality $\vert x+y\vert_K=\vert x\vert_K$, if $\vert y\vert_K < \vert x\vert_K$.

The ring of integers $O = \{x\in K\colon \vert x\vert_K\leq 1\}$ contains an ideal $P=\{x\in K\colon \vert x\vert_K <1\}$. There exists such an element $\be$ that $P=\be O$. The quotient ring $\bar{K}=O/P$ is actually a finite field. The absolute value is called normalized, if $\vert \be\vert_K = q^{-1}$, where $q$ is the cardinality of the finite field $O/P$. Unless stated otherwise, this property is assumed. A normalized absolute value takes exactly the values $q^m$, $m\in\ZZ$.

Let $L$ be a finite extension of a local field $K$. An operator  of multiplication in $L$ by an element $\xi$ can be considered as a linear operator on a $K$-vector space. If the linear function $\xi \mapsto \operatorname{Tr}(\xi)$ does not vanish identically, then the extension is called separable. Here $\operatorname{Tr}$ is the trace of a linear operator in a finite-dimensional vector space. All finite extensions of the field of characteristic zero are separable. On the other hand, the notion of separability makes sense also for the finite fields $\bar K, \bar L$.

A finite extension $L$ of a local field $K$ is called unramified, if $\bar L$ is a separable extension of $\bar K$ of the same degree, as $L$ over $K$. Any local field $K$ has a unique (up to the isomorphism) unramified extension of each given degree $n\geq 2$. The cardinality of the residue field $\bar L$ in this extension equals $q^n$ where $q$ is the residue cardinality for $K$.

The unramified extension of degree $n$ has, as a $K$-vector space, a canonical basis consisting of representatives of a basis in $\bar L$ over $\bar K$. Let $x\in L$ have the coefficients $x_1,\ldots,x_n \in K$ of its expansion with respect to the canonical basis. Then the normalized absolute value $|x|_L$ has the representation \cite{T1976}:
\begin{equation*}\Label{28}\
|x|_L=\Big(\max\limits_{1\leq j\leq n} |x_j|_K\Big)^n.
\end{equation*}
The comparison of $|x|_L$ with another absolute value on $L$, $\| x\|_L=|x|^{1/n}_L\equiv\max\limits_{1\leq j \leq n}|x_j|_K$, which extends the absolute value from $K$, shows that the above canonical basis is an orthogonal basis (in the non-Archimedean sense \cite{Sc1984}) in $L$ considered as Banach space with the norm $\|\cdot \|_L$. Therefore this basis defines an isometric linear isomorphism  $U:\ L\to K^n$,  $Ux=(x_1,\ldots ,x_n)$.

The above isomorphism identifies the operator $D^{\al,n}$ defined by (1)  with the operator
\begin{equation}
\label{4.2}
\left( D_L^\gamma u\right) (x)=\frac{1-q^{n\gamma}}{1-q^{-n(\gamma +1)}}\int\limits_L\frac{u(z)-u(x)}{|z-x|_L^{\gamma +1}}\,dz,\quad \gamma =\dfrac{\alpha}n;
\end{equation}
see \cite{K2021} for the details.

Thus, a multi-dimentional operator over a local field $ K$ is equivalent to a one-dimensional operator over a bigger field, the unramified extension $L$. Such properties are absent in real analysis (except some two-dimensional operators admitting complex representation).

\section{The non-Archimedean heat equation}

Let us consider the problem (2).  We begin with the one-dimensional case ($n=1$).

Using the ultrametric property of the absolute value $|\cdot |_K$, we see that for $x\in \OO$,
\begin{equation}
\left( D^{\al,1}u\right) (t,x)=\left( \Delta_\OO^\al\right) u(t,x)+c_{1,\al}u(t,x)\int\limits_{|y|_K>q^N}\frac{dy}{|y|_K^{\al +1}}-c_{1,\al}\int\limits_{|y|_K>q^N}\frac{u(t,y)dy}{|y|_K^{\al +1}}
\end{equation}
where the operator
$$
\left( \Delta_\OO^\al\right) u (t,x)=c_{1,\al}\int\limits_\OO \frac{u(t,x)-u(t,y)}{|x-y|_K^{\al +1}}\,dy,\quad x\in \OO,
$$
called a local Vladimirov-Taibleson operator in \cite{Br}, is an analog of the regional fractional operator \cite{Gu}. The second term on the right in (4) equals $\lambda_N^{(\al)}u(t,x)$ where
$$
\lambda_N^{(\al)}=\dfrac{q-1}{q(1-q^{-\al-1})}q^{-\al N}
$$
(see \cite{K2018}). The third term on the right in (4) equals 0, due to the ``boundary condition'' in (2).

Thus, the problem (2) is reduced to a kind of a Cauchy problem (on the open and closed set $\OO$) for the equation
\begin{equation}
\frac{\partial u(t,x)}{\partial t}+\left(\Delta_\OO^\al u\right)(t,x)+\lambda_N^{(\al)}u(t,x)=0.
\end{equation}
Setting $u(t,x)=e^{-\lambda_N^{(\al)}t}v(t,x)$  we come to the equation
\begin{equation}
\frac{\partial v(t,x)}{\partial t}+\left(\Delta_\OO^\al v\right)(t,x)=0,\quad x\in \OO,
\end{equation}
the simplest case of a heat equation on a compact $p$-adic manifold \cite{Br}.

An orthonormal basis in $L^2(\OO)$ can be constructed of Kozyrev's $p$-adic wavelets $\psi:\ \OO \to \mathbb C$, which are eigenfunctions of $\Delta_\OO^\al$ corresponding to the eigenvalues \cite{Koz,Br}
\begin{equation}
\rho_\psi = c_{1,\al}\left\{ \int\limits_{\OO\setminus B}|y|_K^{-\al -1}dy+(\operatorname{mes}B)^{-\al}\right\},
\end{equation}
where $\psi$ is supported on a subball $B\subset \OO$. Note that the eigenfunctions $\psi$ do not depend on $\al$.

It is shown in \cite{Br} that the heat kernel function (a fundamental solution of the Cauchy problem) corresponding to the equation (6) has the form
\begin{equation}
H_\al (t,x,y)=1+\sum\limits_\psi e^{-t\rho_\psi}\psi (x)\overline{\psi (y)},\quad x\ne y\in \OO,t>0,
\end{equation}
where $\psi$ runs the set of all the above wavelets. The expansion in (8) begins with 1, since the minimal eigenvalue of the operator in (6) equals zero, with the constant eigenfunction.

If $\al >1$, $H_\al$ exists also for $x=y$, $t>0$ and $0<H(t,\cdot,\cdot)\in L^2(\OO\times \OO)$.

Returning to the equation (5) we get the corresponding heat kernel
\begin{equation}
G_\al (t,x,y)=e^{-\lambda_N^{(\al)}t}H_\al (t,x,y)
\end{equation}
where $H_\al$ is given by (8). The above constructions make sense also for the multi-dimensional situation, with an appropriate constant in (9).

\section{Monotonicity}

Now we formulate the main result of this paper.

\begin{theorem}
Suppose that $N\le -1$, $n<\al_0 <\al_1$. Then
$$
G_{\al_0}(t,x,x)>G_{\al_1}(t,x,x)
$$
for all $t>0$, $x\in \OO$.
\end{theorem}

{\bf Proof.} We begin with the case where $n=1$. We have
$$
\frac{\partial}{\partial \al}\lambda_N^{(\al)}=-\frac{q-1}q q^{-\al N}\frac{\log q}{1-q^{-\al -1}}\left( \frac{q^{-\al -1}}{1-q^{-\al -1}}+N\right)>0,
$$
so that $\lambda_N^{(\al)}$ increases in $\al$, as $\al >1$. Thus $\exp (-\lambda_N^{(\al)}t)$ decreases in $\al$.

Similarly,
$$
\frac{\partial}{\partial \al}c_{1,\al}=\frac{q^\al (1-q^{-\al -1})\log q-(q^\al-1)q^{-\al -1}\log q}{ (1-q^{-\al -1})^2}
$$
and
$$
q^\al (1-q^{-\al -1})-(q^\al -1)q^{-\al -1}=q^\al-2q^{-1}+q^{-\al -1}=q^{-1}(q^{\al +1}-2)+q^{-\al -1}>0,
$$
since $q\ge 2$, so that $c_{1,\al}$ increases in $\al$. It follows from our assumptions regarding $\OO$ and the ultrametric property that the expression in braces in (7) increases in $\alpha$. Thus the eigenvalues $\rho_\psi$ increase, while the expressions $\exp\{-t\rho_\psi\}$ decrease in $\al$, which implies the decrease of the heat kernel.

Turning to the multi-dimensional case, we notice that the isometric isomorphism $U:\ L\to K^n$ transforms the heat kernel of the operator over $K^n$ into a similar one-dimensional operator over the unramified
extension $L$, so that the multi-dimensional case is a consequence of the one-dimensional one considered for a bigger field. $\blacksquare$

\bigskip
\section {Acknowledgements}
The authors are grateful to the anonymous referee for helpful remarks. The first author acknowledges the funding support in the framework of the project ``Spectral Optimization: From Mathematics to Physics and Advanced Technology'' (SOMPATY) received from the European Union’s Horizon 2020 research and innovation programme under the Marie Skłodowska-Curie grant agreement No 873071. The work of the second author was supported by a grant from the Simons Foundation (1030291,A.N.K.).

 %%%%%%%%%%%%%%%%%%%%%%%%%%%%%%

\end{document}